\documentclass[twoside,bezier,12pt, reqno]{amsart}
\usepackage{amssymb,latexsym,epsfig,mathrsfs,graphpap,graphics,amsmath,amssymb}
\usepackage{amstext}
\usepackage{amsthm}
\usepackage[utf8]{inputenc}
\usepackage[english]{babel}
\usepackage{helvet}
\usepackage{courier}
\usepackage{tikz}
\usepackage{enumerate}
\usepackage{graphicx}
\usepackage{float}
\usepackage{subfigure}
\usepackage{framed}
\usetikzlibrary{shapes,arrows}
\usepackage{bm}

\newtheorem{theorem}{Theorem}[section]

\theoremstyle{Definition}
\newtheorem{definition}{Definition}[section]
\newtheorem{example}{Example}[section]

\theoremstyle{remark}
\newtheorem{remark}[theorem]{Remark}
\numberwithin{equation}{section}

\begin{document}

\begin{flushleft}
 {\bf\Large {Discrete Quaternion Quadratic Phase Fourier Transform}}

\parindent=0mm \vspace{.2in}

{\bf{Aamir H. Dar$^{1},$ }}
\end{flushleft}

{{\it $^{1}$ Department of  Mathematical Sciences,  Islamic University of Science and Technology, Kashmir 192122, India.E-mail: $\text{ahdkul740@gmail.com}$}}


\begin{quotation}
\noindent
{\footnotesize {\sc Abstract.}
A novel addition to the family of integral transforms, the quadratic phase Fourier transform (QPFT)  embodies a variety of signal processing tools, including the Fourier transform (FT), fractional Fourier transform (FRFT), linear canonical transform (LCT), and special affine Fourier transforms. Due to its additional degrees of freedom, QPFT performs better in applications than other time-frequency analysis methods. Recently, quaternion quadratic phase Fourier (QQPFT), an extension of the QPFT in quaternion algebra, has been derived and since received noticeable attention because of its expressiveness and grace in the analysis of multi-dimensional quaternion-valued signals and visuals. To the best of our knowledge, the discrete form of the QQPFT is undefined, making it impossible to compute the QQPFT using digital techniques at this time. It initiated us to introduce the two-dimensional (2D) discrete quaternion quadratic phase Fourier (DQQPFT) that is analogous to the 2D discrete quaternion Fourier transform (DQFT). Some fundamental properties including Modulation, the  reconstruction formula and the Plancherel theorem  of the 2D DQQPFT are obtained. Crucially, the fast computation algorithm and convolution theorem  of 2D DQQPFT—which are essential for engineering applications—are also  taken into account. Finally, we present an application of the DQQPFT to study the two-dimensional discrete linear
time-varying systems.\\

}
\end{quotation}

\section{Introduction} \label{sec intro}
\noindent

Using the theory of replicating kernels, Saitoh \cite{wz1} created the quadratic phase Fourier transform (QPFT), an extreme version of the classical Fourier transform (FT), while working on the solution to the heat equation. Motivated by Saitoh work Castro et al. \cite{wz2}, investigated other options for the QPFT by using a general quadratic function in the exponent of the innovative integral transform. More precisely, quadratic phase Fourier transform is a generalization of well known integral transforms whose kernel is in exponential form such as FT, fractional Fourier transform(FRFT), and linear canonical transform (LCT)\cite{amirqpwlt}. Due to the presence of extra arbitrary real parameters, the QPFT exhibits sufficient flexibility and is crucial for solving problems requiring more degrees of freedom. The QPFT, which is a generalization of the well-known Fourier transform, gained popularity periodically and had a significant impact on a number of scientific and engineering fields, including harmonic analysis, optics, pattern recognition, differential equations, quantum mechanics, and more \cite{q4}-\cite{q7}. The processing of discrete data from a digital camera requires the use of numerical methods for approximating the QPFT in real applications. For the first time, Srivastava et al. discretized the QPFT in 2022 \cite{DQPFT}. Following the introduction of the discrete quadratic phase Fourier transform (DQPFT), which is an extension of  the discrete Fourier transform (DFT), the applications of the QPFT grow exponentially(see references \cite{q8a}-\cite{q9} ).\\

It has become common practice to extend integral transforms from the real and complex domains to the quaternion domain in order to examine higher dimensional signals. The the quaternion Fourier transform (QFT), quaternion fractional Fourier transform (QFRFT), the quaternion special affine Fourier transform (QSAFT) and the quaternion linear canonical transform (QLCT)\cite{q10}-\cite{q17} are among the significant ones. With its applications in edge detection, pattern identification, watermarking, color image processing, and image filtering, quaternion algebra has emerged as a prominent field of study in recent decades (see previous publications \cite{22x}-\cite{27x}). In 2023 Bhat and Dar \cite{qqpft} proposed the quaternion quadratic-phase Fourier transform (QQPFT), which is an extension of the QPFT by utilizing the quaternion algebra. In the encoding of quaternion (hypercomplex) signals, QQPFT is important because it generalizes the various quaternion integral transforms like QFT, QFRFT and QLCT. There are many uses for these quaternion integral transforms in numerous fields, including partial differential systems, mathematical statistics, stenography systems,color image processing, and speech recognition \cite{qqpft1}-\cite{qqpft4}.\\

The idea of discrete quaternion Fourier transforms (DQFT) has become one of the most important ideas in digital signal processing since the majority of real-world data are processed at discrete samples. For example, in audio-visual processing, continuous signals are first sampled at discrete time intervals, and then the sampled signal is broken down into its basic periodic components of complex exponentials using Fourier analysis \cite{dd1}-\cite{DQFT2}. The formulation of discrete versions of the  QLCT \cite{DQLCT} is one of the major advancements in discrete Fourier transform development that have occurred recently. The aforementioned advancements, along with the fact that QQPFT theory in quaternion setting is still in its early stages, serve as a catalyst for the construction of a discrete analogue of the QQPFT. Taking this opportunity, our major objective is to present the idea of discrete quaternion quadratic-phase Fourier transform (DQQPFT) and examine both its basic characteristics and practical uses. The majority of DQQPFT's characteristics come from QPFT and discrete quaternion Fourier transform (DQFT). We have taken into consideration the DQQPFT's fast algorithm and convolution theorem to support its application in engineering. Thanks to this work, DQQPFT may be used to solve a variety of signal processing issues.\\

The following lists the article's highlights.
\begin{itemize}
  \item To propose novel transform coined as DQQPFT.
  \item To establish relationship between the proposed DQQPFT and the classical  DQFT. 
  \item To study the important properties including modulation, reconstruction formula and Plancherel theorem of the proposed DQQPFT. 
  \item To consider  convolution theorem and fast computation algorithm of the proposed DQQPFT.
   \item The proposed  DQQPFT is used to study the 2D discrete linear
time-varying (LTV) systems.  
\end{itemize}
This article is structured as follows: A summary of the fundamental characteristics of quaternion algebra and an outline of QPFT definitions are provided in Section \ref{sec2}.  Section \ref{sec3}  presents  the definition of  DQQPFT  and examines its properties  in detail. The convolution theorem and fast algorithm of DQLCT are given  in Section \ref{sec4}. In Section \ref{sec5}, the DQQPFT is used to the analysis of 2-D discrete linear time-varying systems. Finally, a conclusion is drawn in Section \ref{sec6}.
\section{Preliminaries}\label{sec2}
In this section, we provide a quick overview of the quaternion algebra and highlight key concepts, findings, and terminology related to QFT and its quaternion variant, which will be utilized frequently in this study.
\subsection{Quaternion algebra($\mathbb H$)}\,\\
Hamilton introduced the idea of quaternion algebra, which is represented by the letter $\mathbb H$ and is a 4D algebra that is an extension of a complex field $\mathbb C$. Mathematically, it can be expressed as
\begin{equation}\label{2}
\mathbb H=\{f=[f]_0+i[f]_1+j[f]_2+k[f]_3;\quad [f]_s\in \mathbb R, s=0,1,2,3\},
\end{equation}
the three distinct imaginary elements  \{i, j, k\} present in (\ref{2})  follows Hamilton's multiplication rules: $i^2=j^2=k^2=-1=ijk,ij=-ji=k.$\\
For every  quaternion $f\in\mathbb H,$ its conjugate  is given by
\begin{equation}\label{2.0}
\overline{f}=[f]_0-i[f]_1-j[f]_2-k[f]_3.
\end{equation}
Also the anti-involution property is given by\\
\begin{equation}\label{2.1}
\overline{fg}=\overline{g}\overline{f},\quad\overline{\bar{f}}={f},\quad \overline{f+g}=\overline{f}+\overline{g}.
\end{equation}
Let $[f]_0$ denote the real scalar part and $\mathbf{f}=i[f]_1+j[f]_2+k[f]_3$ denote the vector part of quaternion function $f\in\mathbb H$, respectively. Then the real scalar part has a cyclic multiplication symmetry
\begin{equation*}
[fgh]_0=[hfg]_0=[ghf]_0,\quad\forall f,g,h\in\mathbb H.
\end{equation*}
\subsection{Quadratic phase Fourier Transform}\,\\
The QPFT is a neoteric integral transform with an exponential kernel and five real parameters. It offers a straightforward and perceptive unified approach of both transient and non-transient signals. We briefly review earlier studies on the QPFT in this subsection before defining the DQPFT. Next, we refer to the definition of the 2D QQPFT, which is the quaternion algebraic generalization of QPFT.
\begin{definition}[QPFT]For a given real parametric set $\mathfrak{B}=(a,b,c,d,e),$ $b\ne0$, the QPFT of any signal $f(t)\in L^2(\mathbb R)$ is given by \cite{amirqpwlt}
\begin{equation}\label{eqn qpft}
\mathcal Q_{\mathfrak{B}}[f](u)=\frac{1}{\sqrt{2\pi}}\int_{\mathbb R}f(t)\mathcal K_\mathfrak{B}(t,u)dt
\end{equation}
where $\mathcal K_\mathfrak{B}(t,u)=e^{-i(at^2+btu+cu^2+dt+eu)} $ is kernel signal.\\
\end{definition}

\begin{definition}[DQPFT]The DQPFT of any signal $f(\xi)$ corresponding to a parametric
set $\mathfrak{B}=(a,b,c,d,e),$ $b\ne0$, is defined as \cite{DQPFT}
\begin{equation}\label{eqn qpft}
\mathcal Q_{\mathfrak{B}}[f](\omega)=\frac{1}{\sqrt{N}}\sum_{\xi=0}^{N-1}f(\xi)\mathcal K_\mathfrak{B}(\xi,\omega)
\end{equation}
where $\mathcal K_\mathfrak{B}(\xi,\omega)=e^{-i(a\xi^2\vartriangle t^2+\frac{2\pi}{N} \xi\omega+c\omega^2\vartriangle u^2+d\xi\vartriangle t+e\omega\vartriangle u)}$
\end{definition}

\begin{definition}[QQPFT]Consider $\mathfrak{B}_s=(a_s,b_s,c_s,d_s,e_s),b_s\ne0$ and $s=1,2$, then the 2D two sided QQPFT of any quaternion-valued  signal $f(t_1,t_2)$ is denoted by $\mathcal Q^{\mathbb H}_{\mathfrak{B}_1\mathfrak{B}_2}[f](u_1,u_2)$ and defined as \cite{qqpft}
\begin{equation}\label{eqn qpft}
\mathcal Q^{\mathbb H}_{\mathfrak{B}_1\mathfrak{B}_2}[f](u_1,u_2)=\frac{1}{{2\pi}}\int_{\mathbb R^2}\mathcal K^{i}_{\mathfrak{B}_1}(t_1,u_1)f(t_1,t_2)\mathcal K^{j}_{\mathfrak{B}_2}(t_2,u_2)dt_1dt_2
\end{equation}
where
\begin{equation}
\mathcal K^{i}_{\mathfrak{B}_1}(t_1,u_1)=e^{-i(a_1t_1^2+b_1t_1u_1+c_1u_1^2+d_1t_1+e_1u_1)}
\end{equation}
\begin{equation}
\mathcal K^{j}_{\mathfrak{B}_2}(t_2,u_2)=e^{-j(a_2t_2^2+b_2t_2u_2+c_2u_2^2+d_2t_2+e_2u_2)}
\end{equation}
Many well-known linear transforms are included as special instances in the QQPFT (\ref{eqn qpft}) by suitably selecting parameters in $\mathfrak{B}_s=(a_s,b_s,c_s,d_s,e_s), s=1,2$.
\begin{itemize}
  \item For $\mathfrak{B}_s=(-a_s/2b_s,1/b_s,-d_s/2b_s,0,0), s=1,2$, (\ref{eqn qpft}) boils down to  QLCT.
  \item For $\mathfrak{B}_s=(-\cot\theta_s/2,\csc\theta_s,-\cot\theta_s/2,0,0), s=1,2$, (\ref{eqn qpft}) reduces to the QFRFT.
  \item  For $\mathfrak{B}_s=(0,1,0,0,0), s=1,2$, (\ref{eqn qpft}) yields the QFT.
\end{itemize}
\end{definition}

\section{Discrete quaternion quadratic phase Fourier transform and its properties}\label{sec3}
This section presents the DQQPFT of 2D signals and explores its basic characteristics. We also show  how to calculate the DQQPFT and find the DQQPFT-DQFT relationship.

\subsection{Definition of discrete quaternion quadratic phase Fourier transform}\,\\
A signal is assessed at $N_s$ periodic locations in the time domain $t_s$ and the QQPFT domain $u_s,s=1,2$ in order to numerically approximate a quaternion-valued signal $f$ in the QQPFT domain. Consequently, we substitute $t_s=\xi_s$$\vartriangle$$t_s$ and $u_s=\omega_s$$\vartriangle$$u_s$ in the quaternion quadratic phase Fourier transform definition, where $\xi_s$ and $\omega_s$ are integers in $[0,1,2,.....,N_s-1],$ and, respectively, $\vartriangle$$t_s$ and $\vartriangle$$u_s$ are the periodic sampling intervals in the time $t_s,$ space and $u_s$ QQPFT domain, i.e., $\vartriangle$$u_s=2\pi b_s/(N_s$$\vartriangle$$t_s),$ where $s=1,2$. We obtain DQQPFT $\Phi^{\mathbb H}_{\mathfrak{B}_1,\mathfrak{B}_1}(\omega_1,\omega_2)$ of $\phi(\xi_1,\xi_2)=f(\xi_1$$\vartriangle$${t_1},\xi_2$$\vartriangle$$ t_2),$ that is analogous to DFT and DQPFT, by replacing integral with a finite sum
\begin{equation}
\Phi^{\mathbb H}_{\mathfrak{B}_1,\mathfrak{B}_1}(\omega_1,\omega_2)=\sum_{\xi_1=0}^{N_1-1}\sum_{\xi_2=0}^{N_2-1}Z_{\mathfrak{B}_1}(\xi_1,\omega_1)\phi(\xi_1,\xi_2)Z_{\mathfrak{B}_2}(\xi_2,\omega_2)
\end{equation}
where vector $\phi(\xi_1,\xi_2)$ is given by
\begin{equation}
\phi(\xi_1,\xi_2)=\left(
             \begin{array}{c}
               \phi(1,1) \\
               ... \\
               \phi( N_1, N_1) \\
             \end{array}
           \right)
\end{equation}
and the  kernels  $Z_{\mathfrak{B}_1}(\xi_1,\omega_1)$ and $Z_{\mathfrak{B}_2}(\xi_2,\omega_2)$ are $N_1\times N_1$ and $N_2\times N_2$ square DQQPFT matrices and are given by
\begin{equation}
Z_{\mathfrak{B}_1}(\xi_1,\omega_1)=\frac{1}{\sqrt{N_1}}e^{-i(a_1\xi_1^2\vartriangle t_1^2+\frac{2\pi}{N_1} \xi_1\omega_1+c_1\omega_1^2\vartriangle u_1^2+d_1\xi_1\vartriangle t_1+e_1\omega_1\vartriangle u_1)}
\end{equation}
\begin{equation}
Z_{\mathfrak{B}_2}(\xi_2,\omega_2)=\frac{1}{\sqrt{N_2}}e^{-j(a_2\xi_2^2\vartriangle t_2^2+\frac{2\pi}{N_2} \xi_2\omega_2+c_2\omega_2^2\vartriangle u_2^2+d_2\xi_2\vartriangle t_2+e_2\omega_2\vartriangle u_2)}
\end{equation}
where $a_s,b_s,c_s,d_s,e_s$ are real elements of the parametric sets $\mathfrak{B}_s, s=1,2.$\\
Utilizing the above equations, the  two-sided DQQPFT is defined simply as follows:
\begin{definition}[DQQPFT]\label{def dqqpft} Consider $\mathfrak{B}_s=(a_s,b_s,c_s,d_s,e_s),b_s\ne0$ and $s=1,2$, then for 2D signal $f(\xi_1,\xi_2)\in \mathbb H^{N_1\times N_2}$, the two-sided DQQPFT is defined by
\begin{eqnarray}
\nonumber&&\mathcal Q^{\mathbb H}_{\mathfrak{B}_1\mathfrak{B}_2}[f](\omega_1,\omega_2)\\
\nonumber&&=\frac{1}{\sqrt{N_1}}\frac{1}{\sqrt{N_2}}\sum_{\xi_1=0}^{N_1-1}\sum_{\xi_2=0}^{N_2-1}e^{-i(a_1\xi_1^2\vartriangle t_1^2+\frac{2\pi}{N_1} \xi_1\omega_1+c_1\omega_1^2\vartriangle u_1^2+d_1\xi_1\vartriangle t_1+e_1\omega_1\vartriangle u_1)}f(\xi_1,\xi_2)\\
  \label{eqn dqqpft} &&\qquad\qquad\qquad\qquad\qquad\times e^{-j(a_2\xi_2^2\vartriangle t_2^2+\frac{2\pi}{N_2} \xi_2\omega_2+c_2\omega_2^2\vartriangle u_2^2+d_2\xi_2\vartriangle t_2+e_2\omega_2\vartriangle u_2)},
\end{eqnarray}
where $a_s,b_s,c_s,d_s,e_s\in \mathbb R.$
\end{definition}
Definition \ref{def dqqpft} permits us to offer the subsequent observations:
\begin{itemize}
\item The non-commutativity of quaternion-valued functions prevents the interchangeability of the exponential terms found in (\ref{eqn dqqpft}).\\
    
    \item For $\mathfrak{B}_s=(-a_s/2b_s,1/b_s,-d_s/2b_s,0,0), s=1,2$, Definition \ref{def dqqpft} boils down to  DQLCT \cite{DQLCT}.\\
        
  \item For $\mathfrak{B}_s=(-\cot\theta_s/2,\csc\theta_s,-\cot\theta_s/2,0,0), s=1,2$, Definition \ref{def dqqpft} reduces to the discrete QFRFT.\\
      
  \item  For $\mathfrak{B}_s=(0,1,0,0,0), s=1,2$,  Definition \ref{def dqqpft}  yields the DQFT \cite{DQFT1}.\\
      
    \item Similar formulas for the left- and right-sided DQQPFTs may be created by inserting the product of $e^{-i(a_1\xi_1^2\vartriangle t_1^2+\frac{2\pi}{N_1} \xi_1\omega_1+c_1\omega_1^2\vartriangle u_1^2+d_1\xi_1\vartriangle t_1+e_1\omega_1\vartriangle u_1)}$ and \\
        $e^{-j(a_2\xi_2^2\vartriangle t_2^2+\frac{2\pi}{N_2} \xi_2\omega_2+c_2\omega_2^2\vartriangle u_2^2+d_2\xi_2\vartriangle t_2+e_2\omega_2\vartriangle u_2)}$ on the left or right side of $f(\xi_1,\xi_2)$.
    
\end{itemize}

\begin{remark}Since every quaternion function can be written as sum of two complex functions as
\begin{eqnarray*}
{f}&=&\left([f]_0+i[f]_1\right)+j\left([f]_2-i[f]_3\right)\\
&=&\tilde f+j\hat f, 
\end{eqnarray*}
where $\tilde f=[f]_0+i[f]_1$ and $\hat f=[f]_2-i[f]_3$, therefore Definition \ref{def dqqpft} becomes
\begin{equation*}
\mathcal Q^{\mathbb H}_{\mathfrak{B}_1\mathfrak{B}_2}[f](\omega_1,\omega_2)=\mathcal Q^{\mathbb H}_{\mathfrak{B}_1\mathfrak{B}_2}[\tilde f](\omega_1,\omega_2)+j\mathcal Q^{\mathbb H}_{\mathfrak{B}_1\mathfrak{B}_2}[\hat f](\omega_1,\omega_2).
\end{equation*}
Thus it is observed that the DQQPFT of the  signal $f(\xi_1,\xi_2)$ reduces to the total of the DQQPFT of the two complex-valued functions.
\end{remark}
Next, we derive a relationship between the DQFT and the proposed two-sided DQQPFT.

{\bf Relationship with discrete quaternion Fourier transform}\,\\
The Definition \ref{def dqqpft}, can be rewritten as
{\footnotesize\begin{eqnarray}\label{rel with DQFT}
\nonumber&&\mathcal Q^{\mathbb H}_{\mathfrak{B}_1\mathfrak{B}_2}[f](\omega_1,\omega_2)\\
\nonumber&&=\frac{1}{\sqrt{N_1}}\frac{1}{\sqrt{N_2}}\sum_{\xi_1=0}^{N_1-1}\sum_{\xi_2=0}^{N_2-1}e^{-i(a_1\xi_1^2\vartriangle t_1^2+\frac{2\pi}{N_1} \xi_1\omega_1+c_1\omega_1^2\vartriangle u_1^2+d_1\xi_1\vartriangle t_1+e_1\omega_1\vartriangle u_1)}f(\xi_1,\xi_2)\\
   \nonumber&&\qquad\qquad\qquad\qquad\qquad\times e^{-j(a_2\xi_2^2\vartriangle t_2^2+\frac{2\pi}{N_2} \xi_2\omega_2+c_2\omega_2^2\vartriangle u_2^2+d_2\xi_2\vartriangle t_2+e_2\omega_2\vartriangle u_2)}\\
   \nonumber&&=\frac{1}{\sqrt{N_1}}\frac{1}{\sqrt{N_2}}e^{-i(c_1\omega_1^2\vartriangle u_1^2+e_1\omega_1\vartriangle u_1)}\left(\sum_{\xi_1=0}^{N_1-1}\sum_{\xi_2=0}^{N_2-1}e^{-i(a_1\xi_1^2\vartriangle t_1^2+\frac{2\pi}{N_1} \xi_1\omega_1+d_1\xi_1\vartriangle t_1)}f(\xi_1,\xi_2)\right.\\
   \nonumber&&\qquad\qquad\qquad\qquad\qquad\left.\times e^{-j(a_2\xi_2^2\vartriangle t_2^2+\frac{2\pi}{N_2} \xi_2\omega_2+d_2\xi_2\vartriangle t_2)}\right)e^{-j(c_2\omega_2^2\vartriangle u_2^2+e_2\omega_2\vartriangle u_2)}\\
   \nonumber&&=\frac{1}{\sqrt{N_1}}\frac{1}{\sqrt{N_2}}e^{-i(c_1\omega_1^2\vartriangle u_1^2+e_1\omega_1\vartriangle u_1)}\left(\sum_{\xi_1=0}^{N_1-1}\sum_{\xi_2=0}^{N_2-1}e^{-i(a_1\xi_1^2\vartriangle t_1^2+d_1\xi_1\vartriangle t_1)}\left(e^{-i\frac{2\pi}{N_1} \xi_1\omega_1}f(\xi_1,\xi_2)e^{-j\frac{2\pi}{N_2} \xi_2\omega_2}\right)\right.\\
   \nonumber&&\qquad\qquad\qquad\qquad\qquad\left.\times e^{-j(a_2\xi_2^2\vartriangle t_2^2+d_2\xi_2\vartriangle t_2)}\right)e^{-j(c_2\omega_2^2\vartriangle u_2^2+e_2\omega_2\vartriangle u_2)}\\
   \nonumber&&=\frac{1}{\sqrt{N_1}}\frac{1}{\sqrt{N_2}}e^{-i(c_1\omega_1^2\vartriangle u_1^2+e_1\omega_1\vartriangle u_1)}\left(\sum_{\xi_1=0}^{N_1-1}\sum_{\xi_2=0}^{N_2-1}e^{-i\frac{2\pi}{N_1} \xi_1\omega_1}\left(e^{-i(a_1\xi_1^2\vartriangle t_1^2+d_1\xi_1\vartriangle t_1)}f(\xi_1,\xi_2)e^{-j(a_2\xi_2^2\vartriangle t_2^2+d_2\xi_2\vartriangle t_2)}\right)\right.\\
   \nonumber&&\qquad\qquad\qquad\qquad\qquad\left.\times e^{-j\frac{2\pi}{N_2} \xi_2\omega_2}\right)e^{-j(c_2\omega_2^2\vartriangle u_2^2+e_2\omega_2\vartriangle u_2)}\\
     \nonumber&&=\frac{1}{\sqrt{N_1}}\frac{1}{\sqrt{N_2}}e^{-i(c_1\omega_1^2\vartriangle u_1^2+e_1\omega_1\vartriangle u_1)}\mathcal F^{\mathbb H}[g(\xi_1,\xi_2)](\omega_1,\omega_2)e^{-j(c_2\omega_2^2\vartriangle u_2^2+e_2\omega_2\vartriangle u_2)}\\
   \end{eqnarray}
Where

$\mathcal F^{\mathbb H}[g](\omega_1,\omega_2)=\sum_{\xi_1=0}^{N_1-1}\sum_{\xi_2=0}^{N_2-1}e^{-i\frac{2\pi}{N_1} \xi_1\omega_1}g(\xi_1,\xi_2)e^{-j\frac{2\pi}{N_2} \xi_2\omega_2}$ represents two-sided DQFT of function
$g(\xi_1,\xi_2)=e^{-i(a_1\xi_1^2\vartriangle t_1^2+d_1\xi_1\vartriangle t_1)}f(\xi_1,\xi_2)e^{-j(a_2\xi_2^2\vartriangle t_2^2+d_2\xi_2\vartriangle t_2)}$}\\

From (\ref{rel with DQFT}), we see that the computation of the DQQPFT corresponds to the
following steps:\\
(i)A product by a quaternion chirp signal, that is,\\$f(\xi_1,\xi_2)\rightarrow g(\xi_1,\xi_2)=e^{-i(a_1\xi_1^2\vartriangle t_1^2+d_1\xi_1\vartriangle t_1)}f(\xi_1,\xi_2)e^{-j(a_2\xi_2^2\vartriangle t_2^2+d_2\xi_2\vartriangle t_2)}$\\

(ii)A classical discrete quaternion Fourier transform, that is $g(\xi_1,\xi_2)\rightarrow\mathcal F^{\mathbb H}[g](\omega_1,\omega_2)$\\

(iii)Another product by a quaternion chirp signal, i.e.,\\
$ \mathcal F^{\mathbb H}[g](\omega_1,\omega_2)\rightarrow \mathcal Q^{\mathbb H}_{\mathfrak{B}_1\mathfrak{B}_2}[f](\omega_1,\omega_2)=\frac{1}{\sqrt{N_1}}e^{-i(c_1\omega_1^2\vartriangle u_1^2+e_1\omega_1\vartriangle u_1)}\mathcal F^{\mathbb H}[g(\xi_1,\xi_2)](\omega_1,\omega_2)\frac{1}{\sqrt{N_2}}e^{-j(c_2\omega_2^2\vartriangle u_2^2+e_2\omega_2\vartriangle u_2)}.$\\

The aforementioned scheme is depicted in Figure \ref{fig 1}.\\

Now, we shall present a straightforward example that demonstrates how to calculate the DQQPFT of a real image with size of $2\times 2$
pixels.
\begin{example} Given a real image $f=\left[
                                    \begin{array}{cc}
                                      35 & 30 \\
                                      25 & 20 \\
                                    \end{array}
                                  \right].
$ Then for the parameter sets $\mathfrak{B}_1=(0,1-2,2,0)$ and $\mathfrak{B}_1=(0,-1,1,3,0)$ the DQQPFT of $f$ can be computed as:
 \begin{eqnarray*}
 \mathcal Q^{\mathbb H}_{\mathfrak{B}_1\mathfrak{B}_2}[f](0,0)=\frac{1}{2}(35+30+25+20)=55
 \end{eqnarray*}
\begin{eqnarray*}
 \mathcal Q^{\mathbb H}_{\mathfrak{B}_1\mathfrak{B}_2}[f](0,1)&=&\frac{1}{2}\left(f(0,0)-f(0,1)+f(1,0)-f(1,1)\right)\\
 &=&\frac{1}{2}(35-30+25-20)=5
 \end{eqnarray*}
 \begin{eqnarray*}
 \mathcal Q^{\mathbb H}_{\mathfrak{B}_1\mathfrak{B}_2}[f](1,0)&=&\frac{1}{2}\left(f(0,0)+f(0,1)-f(1,0)-f(1,1)\right)\\
 &=&\frac{1}{2}(35+30-25-20)=10
 \end{eqnarray*}
 \begin{eqnarray*}
 \mathcal Q^{\mathbb H}_{\mathfrak{B}_1\mathfrak{B}_2}[f](1,1)&=&\frac{1}{2}\left(f(0,0)-f(0,1)-f(1,0)+f(1,1)\right)\\
 &=&\frac{1}{2}(35-30-25+20)=0
 \end{eqnarray*}
 Hence the DQQPFT of $f$ is $\left[
                                    \begin{array}{cc}
                                      55 & 5 \\
                                      10 & 0 \\
                                    \end{array}
                                  \right].$
\end{example}

\subsection{Properties}\,\\
In this subsection our investigation shall focus on the fundamental properties of the Discrete version of QQPFT defined in (\ref{eqn dqqpft}).
\begin{theorem}[Linearity] The DQQPFT of two quaternion-valued signals $f,g\in L^2(\mathbb R^2,\mathbb H)$ is given by:
\begin{equation}
 \mathcal Q^{\mathbb H}_{\mathfrak{B}_1\mathfrak{B}_2}[\alpha f+\beta g](\omega_1,\omega_2)= \alpha\mathcal Q^{\mathbb H}_{\mathfrak{B}_1\mathfrak{B}_2}[f](\omega_1,\omega_2)+\beta\mathcal Q^{\mathbb H}_{\mathfrak{B}_1\mathfrak{B}_2}[g](\omega_1,\omega_2),\quad \alpha,\beta\in \mathbb R.
\end{equation}

\end{theorem}
\begin{proof}
Follows from definition of DQQPFT, so we avoid it.
\end{proof}

\begin{theorem}[Conjugate]For a signal $f(\xi_1,\xi_2)=f_0(\xi_1,\xi_2)+if_1(\xi_1,\xi_2)+jf_2(\xi_1,\xi_2)+kf_3(\xi_1,\xi_2),$ we have
\begin{eqnarray}
 \nonumber&&\mathcal Q^{\mathbb H}_{\mathfrak{B}_1\mathfrak{B}_2}[\overline{f(\xi_1,\xi_2)}](\omega_1,\omega_2)\\
  \nonumber&&=\mathcal Q^{\mathbb H}_{\mathfrak{B}_1\mathfrak{B}_2}[{f_0}](\omega_1,\omega_2)-i.\mathcal Q^{\mathbb H}_{\mathfrak{B}_1\mathfrak{B}_2}[{f_1}](\omega_1,\omega_2)-\mathcal Q^{\mathbb H}_{\mathfrak{B}_1\mathfrak{B}_2}[{f_2}](\omega_1,\omega_2).j-i.\mathcal Q^{\mathbb H}_{\mathfrak{B}_1\mathfrak{B}_2}[{f_3}](\omega_1,\omega_2).k.\\
 \end{eqnarray}
\end{theorem}

\begin{proof}
Using $k=i.j$ in the definition of DQQPFT, the proof follows.
\end{proof}

\begin{theorem}[Translation] For a quaternion-valued signal $f\in \mathbb H^{N_1\times N_2},$ we have
\begin{eqnarray}
 \nonumber&&\mathcal Q^{\mathbb H}_{\mathfrak{B}_1\mathfrak{B}_2}[{f(\xi_1-k_1,\xi_2-k_2)}](\omega_1,\omega_2)\\
 \nonumber&&=e^{-i(a_1k_1^2\vartriangle t_1^2+\frac{2\pi}{N_1} k_1\omega_1+d_1k_1\vartriangle t_1)}\mathcal Q^{\mathbb H}_{\mathfrak{B}_1\mathfrak{B}_2}[e^{-2ia_1\epsilon_1k_1\vartriangle t_1^2}f(\epsilon_1,\epsilon_2)e^{-2ja_2\epsilon_2k_2\vartriangle t_2^2}](\omega_1,\omega_2)\\
  \nonumber&&\qquad\qquad\qquad\qquad\qquad\qquad\qquad\qquad\qquad\qquad\qquad\times e^{-j(a_2k_2^2\vartriangle t_2^2+\frac{2\pi}{N_2} k_2\omega_2+d_2k_2\vartriangle t_2)}\\
\end{eqnarray}
where $k_s\in N_s,$ $s=1,2.$ 
\end{theorem}
\begin{proof}
From Definition \ref{def dqqpft}, we have
\begin{eqnarray*}
 \nonumber&&\mathcal Q^{\mathbb H}_{\mathfrak{B}_1\mathfrak{B}_2}[{f(\xi_1-k_1,\xi_2-k_2)}](\omega_1,\omega_2)\\
 \nonumber&&=\frac{1}{\sqrt{N_1}}\frac{1}{\sqrt{N_2}}\sum_{\xi_1=0}^{N_1-1}\sum_{\xi_2=0}^{N_2-1}e^{-i(a_1\xi_1^2\vartriangle t_1^2+\frac{2\pi}{N_1} \xi_1\omega_1+c_1\omega_1^2\vartriangle u_1^2+d_1\xi_1\vartriangle t_1+e_1\omega_1\vartriangle u_1)}f(\xi_1-k_1,\xi_2-k_2)\\
   \nonumber&&\qquad\qquad\qquad\qquad\qquad\times e^{-j(a_2\xi_2^2\vartriangle t_2^2+\frac{2\pi}{N_2} \xi_2\omega_2+c_2\omega_2^2\vartriangle u_2^2+d_2\xi_2\vartriangle t_2+e_2\omega_2\vartriangle u_2)}\\
\nonumber&&=\frac{1}{\sqrt{N_1}}\frac{1}{\sqrt{N_2}}\sum_{\epsilon_1=0}^{N_1-1}\sum_{\epsilon_2=0}^{N_2-1}e^{-i[a_1(\epsilon_1+k_1)^2\vartriangle t_1^2+\frac{2\pi}{N_1} (\epsilon_1+k_1)\omega_1+c_1\omega_1^2\vartriangle u_1^2+d_1(\epsilon_1+k_1)\vartriangle t_1+e_1\omega_1\vartriangle u_1]}f(\epsilon_1,\epsilon_2)\\
   \nonumber&&\qquad\qquad\qquad\qquad\qquad\times e^{-j[a_2(\epsilon_2+k_2)^2\vartriangle t_2^2+\frac{2\pi}{N_2} (\epsilon_2+k_2)\omega_2+c_2\omega_2^2\vartriangle u_2^2+d_2(\epsilon_2+k_2)\vartriangle t_2+e_2\omega_2\vartriangle u_2]}\\
   \nonumber&&=e^{-i(a_1k_1^2\vartriangle t_1^2+\frac{2\pi}{N_1} k_1\omega_1+d_1k_1\vartriangle t_1)}\frac{1}{\sqrt{N_1}}\frac{1}{\sqrt{N_2}} \sum_{\epsilon_1=0}^{N_1-1}\sum_{\epsilon_2=0}^{N_2-1}\left\{e^{-i(a_1\epsilon_1^2\vartriangle t_1^2+\frac{2\pi}{N_1} \epsilon_1\omega_1+c_1\omega_1^2\vartriangle u_1^2+d_1\epsilon_1\vartriangle t_1+e_1\omega_1\vartriangle u_1)}\right.\\
     \nonumber&&\qquad\times\left. \left(e^{-2ia_1\epsilon_1k_1\vartriangle t_1^2}f(\epsilon_1,\epsilon_2)e^{-2ja_2\epsilon_2k_2\vartriangle t_2^2}\right)e^{-j(a_2\epsilon_2^2\vartriangle t_2^2+\frac{2\pi}{N_2} \epsilon_2\omega_2+c_2\omega_2^2\vartriangle u_2^2+d_2\epsilon_2\vartriangle t_2+e_2\omega_2\vartriangle u_2)}\right\}\\
   \nonumber&&\qquad\qquad\qquad\qquad\qquad\qquad\qquad\qquad\qquad\qquad\times e^{-j(a_2k_2^2\vartriangle t_2^2+\frac{2\pi}{N_2} k_2\omega_2+d_2k_2\vartriangle t_2)}\\
  \nonumber&&=e^{-i(a_1k_1^2\vartriangle t_1^2+\frac{2\pi}{N_1} k_1\omega_1+d_1k_1\vartriangle t_1)}\\
   \nonumber&&\qquad\qquad\qquad\qquad\qquad\times \mathcal Q^{\mathbb H}_{\mathfrak{B}_1\mathfrak{B}_2}[e^{-2ia_1\epsilon_1k_1\vartriangle t_1^2}f(\epsilon_1,\epsilon_2)e^{-2ja_2\epsilon_2k_2\vartriangle t_2^2}](\omega_1,\omega_2)\\
  \nonumber&&\qquad\qquad\qquad\qquad\qquad\qquad\qquad\qquad\qquad\times e^{-j(a_2k_2^2\vartriangle t_2^2+\frac{2\pi}{N_2} k_2\omega_2+d_2k_2\vartriangle t_2)}.\\
\end{eqnarray*}
This completes the proof.
\end{proof}
\begin{theorem}[Modulation] For a quaternion-valued modulated signal $e^{i\frac{2\pi}{N_1} \epsilon_1\xi_1}f(\xi_1,\xi_2)e^{j\frac{2\pi}{N_2} \epsilon_2\xi_2}$ in $\mathbb H^{N_1\times N_2},$ the  DQQPFT takes the following form
\begin{eqnarray}
 \nonumber&&\mathcal Q^{\mathbb H}_{\mathfrak{B}_1\mathfrak{B}_2}[e^{i\frac{2\pi}{N_1} \epsilon_1\xi_1}f(\xi_1,\xi_2)e^{j\frac{2\pi}{N_2} \epsilon_2\xi_2}](\omega_1,\omega_2)\\
 \nonumber&&=e^{i[c_1(\epsilon_1^2-2\omega_1\epsilon_1)\vartriangle u_1^2-e_1\epsilon_1\vartriangle u_1]}\\
   \nonumber&&\qquad\qquad\qquad\qquad\qquad\times\mathcal Q^{\mathbb H}_{\mathfrak{B}_1\mathfrak{B}_2}[f(\xi_1,\xi_2) ](\omega_1-\epsilon_1,\omega_2-\epsilon_2)\\
   \nonumber&&\qquad\qquad\qquad\qquad\qquad\qquad\qquad\qquad\qquad\times e^{j[c_2(\epsilon_2^2-2\omega_2\epsilon_2)\vartriangle u_2^2-e_2\epsilon_2\vartriangle u_2]}\\
\end{eqnarray}
\end{theorem}
\begin{proof}
From definition of DQQPFT, we have
\begin{eqnarray*}
 \nonumber&&\mathcal Q^{\mathbb H}_{\mathfrak{B}_1\mathfrak{B}_2}[e^{i\frac{2\pi}{N_1} \epsilon_1\xi_1}f(\xi_1,\xi_2)e^{j\frac{2\pi}{N_1} \epsilon_1\xi_1}](\omega_1,\omega_2)\\
 \nonumber&&=\frac{1}{\sqrt{N_1}}\frac{1}{\sqrt{N_2}}\sum_{\xi_1=0}^{N_1-1}\sum_{\xi_2=0}^{N_2-1}e^{-i(a_1\xi_1^2\vartriangle t_1^2+\frac{2\pi}{N_1} \xi_1\omega_1+c_1\omega_1^2\vartriangle u_1^2+d_1\xi_1\vartriangle t_1+e_1\omega_1\vartriangle u_1)}e^{i\frac{2\pi}{N_1} \epsilon_1\xi_1}f(\xi_1,\xi_2)e^{j\frac{2\pi}{N_2} \epsilon_2\xi_2}\\
   \nonumber&&\qquad\qquad\qquad\qquad\qquad\times e^{-j(a_2\xi_2^2\vartriangle t_2^2+\frac{2\pi}{N_2} \xi_2\omega_2+c_2\omega_2^2\vartriangle u_2^2+d_2\xi_2\vartriangle t_2+e_2\omega_2\vartriangle u_2)}\\
   \nonumber&&=\frac{1}{\sqrt{N_1}}\frac{1}{\sqrt{N_2}}\sum_{\xi_1=0}^{N_1-1}\sum_{\xi_2=0}^{N_2-1}e^{-i(a_1\xi_1^2\vartriangle t_1^2+\frac{2\pi}{N_1} \xi_1(\omega_1-\epsilon_1)+c_1\omega_1^2\vartriangle u_1^2+d_1\xi_1\vartriangle t_1+e_1\omega_1\vartriangle u_1)}f(\xi_1,\xi_2)\\
   \nonumber&&\qquad\qquad\qquad\qquad\qquad\times e^{-j(a_2\xi_2^2\vartriangle t_2^2+\frac{2\pi}{N_2} \xi_2(\omega_2-\epsilon_2)+c_2\omega_2^2\vartriangle u_2^2+d_2\xi_2\vartriangle t_2+e_2\omega_2\vartriangle u_2)}\\
   \nonumber&&=\frac{1}{\sqrt{N_1}}\frac{1}{\sqrt{N_2}}\sum_{\xi_1=0}^{N_1-1}\sum_{\xi_2=0}^{N_2-1}e^{-i[a_1\xi_1^2\vartriangle t_1^2+\frac{2\pi}{N_1} \xi_1(\omega_1-\epsilon_1)+c_1(\omega_1-\epsilon_1)^2\vartriangle u_1^2+d_1\xi_1\vartriangle t_1+e_1(\omega_1-\epsilon_1)\vartriangle u_1]}\\
  \nonumber&&\qquad\qquad\qquad\qquad\qquad\times  e^{i[c_1(\epsilon_1^2-2\omega_1\epsilon_1)\vartriangle u_1^2-e_1\epsilon_1\vartriangle u_1]}f(\xi_1,\xi_2) e^{j[c_2(\epsilon_2^2-2\omega_2\epsilon_2)\vartriangle u_2^2-e_2\epsilon_2\vartriangle u_2]}\\
   \nonumber&&\qquad\qquad\qquad\qquad\qquad\qquad\qquad\times e^{-j[a_2\xi_2^2\vartriangle t_2^2+\frac{2\pi}{N_2} \xi_2(\omega_2-\epsilon_2)+c_2(\omega_2-\epsilon_2)^2\vartriangle u_2^2+d_2\xi_2\vartriangle t_2+e_2(\omega_2-\epsilon_2)\vartriangle u_2]}\\
   \nonumber&&=e^{i[c_1(\epsilon_1^2-2\omega_1\epsilon_1)\vartriangle u_1^2-e_1\epsilon_1\vartriangle u_1]}\\
  \nonumber&&\qquad\qquad\times \frac{1}{\sqrt{N_1}}\frac{1}{\sqrt{N_2}}\sum_{\xi_1=0}^{N_1-1}\sum_{\xi_2=0}^{N_2-1}\left\{e^{-i[a_1\xi_1^2\vartriangle t_1^2+\frac{2\pi}{N_1} \xi_1(\omega_1-\epsilon_1)+c_1(\omega_1-\epsilon_1)^2\vartriangle u_1^2+d_1\xi_1\vartriangle t_1+e_1(\omega_1-\epsilon_1)\vartriangle u_1]} \right.\\
   \nonumber&&\qquad\qquad\qquad\times \left. f(\xi_1,\xi_2)e^{-j[a_2\xi_2^2\vartriangle t_2^2+\frac{2\pi}{N_2} \xi_2(\omega_2-\epsilon_2)+c_2(\omega_2-\epsilon_2)^2\vartriangle u_2^2+d_2\xi_2\vartriangle t_2+e_2(\omega_2-\epsilon_2)\vartriangle u_2]}\right\}\\
   \nonumber&&\qquad\qquad\qquad\qquad\qquad\qquad\qquad\qquad\qquad\qquad\times e^{j[c_2(\epsilon_2^2-2\omega_2\epsilon_2)\vartriangle u_2^2-e_2\epsilon_2\vartriangle u_2]}\\
   \nonumber&&=e^{i[c_1(\epsilon_1^2-2\omega_1\epsilon_1)\vartriangle u_1^2-e_1\epsilon_1\vartriangle u_1]}\\
   \nonumber&&\qquad\qquad\qquad\qquad\qquad\times\mathcal Q^{\mathbb H}_{\mathfrak{B}_1\mathfrak{B}_2}[f(\xi_1,\xi_2) ](\omega_1-\epsilon_1,\omega_2-\epsilon_2)\\
   \nonumber&&\qquad\qquad\qquad\qquad\qquad\qquad\qquad\qquad\qquad\times e^{j[c_2(\epsilon_2^2-2\omega_2\epsilon_2)\vartriangle u_2^2-e_2\epsilon_2\vartriangle u_2]}.
\end{eqnarray*}
Hence completes the proof.
\end{proof}

In our next theorem, we demonstrate that the discrete QQPFT is reversible in nature.
\begin{theorem}[Reconstruction formula]Let $\mathcal Q^{\mathbb H}_{\mathfrak{B}_1\mathfrak{B}_2}[f](\omega_1,\omega_2)$ denote the discrete QQPFT of any quaternion-valued signal $f$, then $f$ can be reconstructed back by the following formula:
\begin{eqnarray}
\nonumber&&f(\epsilon_1,\epsilon_2)\\
\nonumber&&=\frac{1}{\sqrt{N_1}}\frac{1}{\sqrt{N_2}}\sum_{\epsilon_1=0}^{N_1-1}\sum_{\epsilon_2=0}^{N_2-1}e^{i(a_1\epsilon_1^2\vartriangle t_1^2+\frac{2\pi}{N_1} \epsilon_1\omega_1+c_1\omega_1^2\vartriangle u_1^2+d_1\epsilon_1\vartriangle t_1+e_1\omega_1\vartriangle u_1)}\mathcal Q^{\mathbb H}_{\mathfrak{B}_1\mathfrak{B}_2}[f](\omega_1,\omega_2)\\
\nonumber&&\qquad\qquad\qquad\qquad\qquad\qquad\qquad\qquad\times e^{j(a_2\epsilon_2^2\vartriangle t_2^2+\frac{2\pi}{N_2} \epsilon_2\omega_2+c_2\omega_2^2\vartriangle u_2^2+d_2\epsilon_2\vartriangle t_2+e_2\omega_2\vartriangle u_2)}\\
\end{eqnarray}
\end{theorem}
\begin{proof}
By the virtue of the definition of DQQPFT, we can write

\begin{eqnarray}\label{eqn 3.3}
\nonumber&&\sum_{\epsilon_1=0}^{N_1-1}\sum_{\epsilon_2=0}^{N_2-1}e^{i(a_1\epsilon_1^2\vartriangle t_1^2+\frac{2\pi}{N_1} \epsilon_1\omega_1+c_1\omega_1^2\vartriangle u_1^2+d_1\epsilon_1\vartriangle t_1+e_1\omega_1\vartriangle u_1)}\mathcal Q^{\mathbb H}_{\mathfrak{B}_1\mathfrak{B}_2}[f](\omega_1,\omega_2)\\
\nonumber&&\qquad\qquad\qquad\qquad\qquad\times e^{j(a_2\epsilon_2^2\vartriangle t_2^2+\frac{2\pi}{N_2} \epsilon_2\omega_2+c_2\omega_2^2\vartriangle u_2^2+d_2\epsilon_2\vartriangle t_2+e_2\omega_2\vartriangle u_2)}\\
\nonumber&&=\frac{1}{\sqrt{N_1}}\frac{1}{\sqrt{N_2}}\sum_{\epsilon_1=0}^{N_1-1}\sum_{\epsilon_2=0}^{N_2-1}e^{i(a_1\epsilon_1^2\vartriangle t_1^2+\frac{2\pi}{N_1} \epsilon_1\omega_1+c_1\omega_1^2\vartriangle u_1^2+d_1\epsilon_1\vartriangle t_1+e_1\omega_1\vartriangle u_1)}\\
\nonumber&&\qquad\times\sum_{\xi_1=0}^{N_1-1}\sum_{\xi_2=0}^{N_2-1}e^{-i(a_1\xi_1^2\vartriangle t_1^2+\frac{2\pi}{N_1} \xi_1\omega_1+c_1\omega_1^2\vartriangle u_1^2+d_1\xi_1\vartriangle t_1+e_1\omega_1\vartriangle u_1)}f(\xi_1,\xi_2)\\
   \nonumber&&\qquad\qquad\qquad\qquad\qquad\times e^{-j(a_2\xi_2^2\vartriangle t_2^2+\frac{2\pi}{N_2} \xi_2\omega_2+c_2\omega_2^2\vartriangle u_2^2+d_2\xi_2\vartriangle t_2+e_2\omega_2\vartriangle u_2)}\\
   \nonumber&&\qquad\qquad\qquad\qquad\qquad\qquad\qquad\times e^{j(a_2\epsilon_2^2\vartriangle t_2^2+\frac{2\pi}{N_2} \epsilon_2\omega_2+c_2\omega_2^2\vartriangle u_2^2+d_2\epsilon_2\vartriangle t_2+e_2\omega_2\vartriangle u_2)}\\
   \nonumber&&=\frac{1}{\sqrt{N_1}}\frac{1}{\sqrt{N_2}}\sum_{\epsilon_1=0}^{N_1-1}\sum_{\epsilon_2=0}^{N_2-1}\sum_{\xi_1=0}^{N_1-1}\sum_{\xi_2=0}^{N_2-1}
   e^{-i[a_1(\xi_1^2-\epsilon_1^2)\vartriangle t_1^2+\frac{2\pi}{N_1} (\xi_1-\epsilon_1)\omega_1+d_1(\xi_1-\epsilon_1)\vartriangle t_1]}f(\xi_1,\xi_2)\\
   \nonumber&&\qquad\qquad\qquad\qquad\qquad\qquad\qquad\qquad\qquad\times e^{-j[a_2(\xi_2^2-\epsilon_2^2)\vartriangle t_2^2+\frac{2\pi}{N_2} (\xi_2-\epsilon_2)\omega_2+d_2(\xi_2-\epsilon_2)\vartriangle t_2]}.\\
\end{eqnarray}
Using the sums
\begin{eqnarray*}
\sum_{\epsilon_1=0}^{N_1-1}\sum_{\xi_1=0}^{N_1-1}e^{-i[a_1(\xi_1^2-\epsilon_1^2)\vartriangle t_1^2+\frac{2\pi}{N_1} (\xi_1-\epsilon_1)\omega_1+d_1(\xi_1-\epsilon_1)\vartriangle t_1]}=\left\{\begin{array}{cc}
        N_1, & if \epsilon_1=\xi_1; \\
        0, & if \epsilon_1\ne \xi_1;
      \end{array}\right.
\end{eqnarray*}
\begin{eqnarray*}
\sum_{\epsilon_2=0}^{N_2-1}\sum_{\xi_2=0}^{N_2-1}e^{-j[a_2(\xi_2^2-\epsilon_2^2)\vartriangle t_2^2+\frac{2\pi}{N_2} (\xi_2-\epsilon_2)\omega_2+d_2(\xi_2-\epsilon_2)\vartriangle t_2]}=\left\{\begin{array}{cc}
        N_2, & {\mbox if}\, \epsilon_2=\xi_2; \\
        0, & {\mbox if}\, \epsilon_2\ne \xi_2;
      \end{array}\right.
\end{eqnarray*}
The RHS of (\ref{eqn 3.3}), yields
\begin{eqnarray*}
&&=\frac{1}{\sqrt{N_1}}\frac{1}{\sqrt{N_2}}\sum_{\epsilon_1=0}^{N_1-1}\sum_{\xi_1=0}^{N_1-1}e^{-i[a_1(\xi_1^2-\epsilon_1^2)\vartriangle t_1^2+\frac{2\pi}{N_1} (\xi_1-\epsilon_1)\omega_1+d_1(\xi_1-\epsilon_1)\vartriangle t_1]}f(\xi_1,\epsilon_2)N_2\\
&&=\sqrt{\frac{N_2}{N_1}}N_1f(\epsilon_1,\epsilon_2)\\
&&=\sqrt{{N_2}{N_1}}f(\epsilon_1,\epsilon_2).
\end{eqnarray*}
This completes the proof.
\end{proof}

Towards the culmination of this section, we derive the Plancherel theorem for the
proposed discrete QQPFT which states that the total signal energy computed in the time domain equals to the total signal energy in the frequency domain.

\begin{theorem}[Plancherel]Let $\mathcal Q^{\mathbb H}_{\mathfrak{B}_1\mathfrak{B}_2}[f](\omega_1,\omega_2)$ denote the discrete QQPFT of any quaternion-valued signal $f$, then we have
\begin{eqnarray}
\sum_{\omega_1=0}^{N_1-1}\sum_{\omega_2=0}^{N_1-1}\left|\mathcal Q^{\mathbb H}_{\mathfrak{B}_1\mathfrak{B}_2}[f](\omega_1,\omega_2)\right|^2=\frac{1}{N_1N_2}\sum_{\xi_1=0}^{N_1-1}\sum_{\xi_2=0}^{N_2-1}\left|f(\xi_1,\xi_2)\right|^2.
\end{eqnarray}
\end{theorem}
\begin{proof}
Invoking the definition of DQQPFT, we have
\begin{eqnarray*}
&&\sum_{\omega_1=0}^{N_1-1}\sum_{\omega_2=0}^{N_1-1}\left|\mathcal Q^{\mathbb H}_{\mathfrak{B}_1\mathfrak{B}_2}[f](\omega_1,\omega_2)\right|^2\\
&&=\sum_{\omega_1=0}^{N_1-1}\sum_{\omega_2=0}^{N_1-1}\mathcal Q^{\mathbb H}_{\mathfrak{B}_1\mathfrak{B}_2}[f](\omega_1,\omega_2)\overline{\mathcal Q^{\mathbb H}_{\mathfrak{B}_1\mathfrak{B}_2}[f](\omega_1,\omega_2)}\\
&&=\frac{1}{{N_1}N_2}\sum_{\omega_1=0}^{N_1-1}\sum_{\omega_2=0}^{N_1-1}\sum_{\xi_1=0}^{N_1-1}\sum_{\xi_2=0}^{N_2-1}e^{-i(a_1\xi_1^2\vartriangle t_1^2+\frac{2\pi}{N_1} \xi_1\omega_1+c_1\omega_1^2\vartriangle u_1^2+d_1\xi_1\vartriangle t_1+e_1\omega_1\vartriangle u_1)}f(\xi_1,\xi_2)\\\\
   &&\times e^{-j(a_2\xi_2^2\vartriangle t_2^2+\frac{2\pi}{N_2} \xi_2\omega_2+c_2\omega_2^2\vartriangle u_2^2+d_2\xi_2\vartriangle t_2+e_2\omega_2\vartriangle u_2)} \overline{e^{-i(a_1\xi_1^2\vartriangle t_1^2+\frac{2\pi}{N_1} \xi_1\omega_1+c_1\omega_1^2\vartriangle u_1^2+d_1\xi_1\vartriangle t_1+e_1\omega_1\vartriangle u_1)}}\\\\
  &&\qquad\qquad\qquad\qquad\qquad\times \overline{f(\xi_1,\xi_2)e^{-j(a_2\xi_2^2\vartriangle t_2^2+\frac{2\pi}{N_2} \xi_2\omega_2+c_2\omega_2^2\vartriangle u_2^2+d_2\xi_2\vartriangle t_2+e_2\omega_2\vartriangle u_2)}}\\
  &&=\frac{1}{{N_1}N_2}\sum_{\omega_1=0}^{N_1-1}\sum_{\omega_2=0}^{N_1-1}\sum_{\xi_1=0}^{N_1-1}\sum_{\xi_2=0}^{N_2-1}e^{-i(a_1\xi_1^2\vartriangle t_1^2+\frac{2\pi}{N_1} \xi_1\omega_1+c_1\omega_1^2\vartriangle u_1^2+d_1\xi_1\vartriangle t_1+e_1\omega_1\vartriangle u_1)}f(\xi_1,\xi_2)\\\\
   &&\times e^{-j(a_2\xi_2^2\vartriangle t_2^2+\frac{2\pi}{N_2} \xi_2\omega_2+c_2\omega_2^2\vartriangle u_2^2+d_2\xi_2\vartriangle t_2+e_2\omega_2\vartriangle u_2)} {e^{j(a_2\xi_2^2\vartriangle t_2^2+\frac{2\pi}{N_2} \xi_2\omega_2+c_2\omega_2^2\vartriangle u_2^2+d_2\xi_2\vartriangle t_2+e_2\omega_2\vartriangle u_2)}}\\\\
  &&\qquad\qquad\qquad\qquad\qquad\times \overline{f(\xi_1,\xi_2)}e^{i(a_1\xi_1^2\vartriangle t_1^2+\frac{2\pi}{N_1} \xi_1\omega_1+c_1\omega_1^2\vartriangle u_1^2+d_1\xi_1\vartriangle t_1+e_1\omega_1\vartriangle u_1)}\\
  &&=\frac{1}{{N_1}N_2}\sum_{\xi_1=0}^{N_1-1}\sum_{\xi_2=0}^{N_2-1}e^{-i(a_1\xi_1^2\vartriangle t_1^2+\frac{2\pi}{N_1} \xi_1\omega_1+c_1\omega_1^2\vartriangle u_1^2+d_1\xi_1\vartriangle t_1+e_1\omega_1\vartriangle u_1)}\\\\
  &&\qquad\qquad\qquad\qquad\qquad\qquad\times f(\xi_1,\xi_2) \overline{f(\xi_1,\xi_2)}e^{i(a_1\xi_1^2\vartriangle t_1^2+\frac{2\pi}{N_1} \xi_1\omega_1+c_1\omega_1^2\vartriangle u_1^2+d_1\xi_1\vartriangle t_1+e_1\omega_1\vartriangle u_1)}\\
   &&=\frac{1}{{N_1}N_2}\sum_{\xi_1=0}^{N_1-1}\sum_{\xi_2=0}^{N_2-1}e^{-i(a_1\xi_1^2\vartriangle t_1^2+\frac{2\pi}{N_1} \xi_1\omega_1+c_1\omega_1^2\vartriangle u_1^2+d_1\xi_1\vartriangle t_1+e_1\omega_1\vartriangle u_1)}\\\\
  &&\qquad\qquad\qquad\qquad\qquad\qquad\times \left|f(\xi_1,\xi_2)\right|^2 e^{i(a_1\xi_1^2\vartriangle t_1^2+\frac{2\pi}{N_1} \xi_1\omega_1+c_1\omega_1^2\vartriangle u_1^2+d_1\xi_1\vartriangle t_1+e_1\omega_1\vartriangle u_1)}\\
   &&=\frac{1}{{N_1}N_2}\sum_{\xi_1=0}^{N_1-1}\sum_{\xi_2=0}^{N_2-1}\left|f(\xi_1,\xi_2)\right|^2e^{-i(a_1\xi_1^2\vartriangle t_1^2+\frac{2\pi}{N_1} \xi_1\omega_1+c_1\omega_1^2\vartriangle u_1^2+d_1\xi_1\vartriangle t_1+e_1\omega_1\vartriangle u_1)}\\\\
  &&\qquad\qquad\qquad\qquad\qquad\qquad\times e^{i(a_1\xi_1^2\vartriangle t_1^2+\frac{2\pi}{N_1} \xi_1\omega_1+c_1\omega_1^2\vartriangle u_1^2+d_1\xi_1\vartriangle t_1+e_1\omega_1\vartriangle u_1)}\\
  &&=\frac{1}{{N_1}N_2}\sum_{\xi_1=0}^{N_1-1}\sum_{\xi_2=0}^{N_2-1}\left|f(\xi_1,\xi_2)\right|^2.
\end{eqnarray*}
It is worth to note that  $\left|f(\xi_1,\xi_2)\right|^2$ is a real-valued in second last step.\\
Hence completes proof.
\end{proof}

\section{Convolution and Fast algorithm of discrete quaternion quadratic phase Fourier transform}\label{sec4}
\subsection{Convolution }
Convolution is a mathematical technique in which two signals are combined to create a third signal. Because it links the three signals—the input, output, and impulse response—convolution is a crucial digital signal processing approach. Convolution is simple to implement and computes relatively quickly because to its simplicity. To improve the applications of discrete quaternion quadratic-phase Fourier transform, its convolution must be studied. We develop the discrete quaternion quadratic-phase Fourier transform convolution theorem as a consequence. We start by defining the DQQPFT's convolution.
\begin{definition}Let $f,g$ be two quaternion-valued signals in $L^2(\mathbb R^2, \mathbb H),$ then the discrete quaternion quadratic-phase
convolution is denoted by $\star$ and is defined by
\begin{equation}\label{eqn con}
(f\star g)(\xi_1,\xi_2)=\sum_{z_1=0}^{N_1-1}\sum_{z_2=0}^{N_2-1}e^{-i2a_1z_1(z_1-\xi_1)\vartriangle t_1^2}f(z_1,z_2)g(\xi_1-z_1,\xi_2-z_2)e^{-j2a_2z_2(z_2-\xi_2)\vartriangle t_2^2}
\end{equation}
\end{definition}
The above definition  deduces the following theorem, which demonstrates how the convolution of two quaternion-valued functions interacts with their Discrete QQPFTs
\begin{theorem}\label{thm con} Let $f,g\in L^2(\mathbb R^2,\mathbb H)$ be two quaternion valued signals. Assuming that $\mathcal Q^{\mathbb H}_{\mathfrak{B}_1\mathfrak{B}_2}[f](\omega_1,\omega_2)\in L^2(\mathbb R^2,\mathbb H)$ and $\mathcal Q^{\mathbb H}_{\mathfrak{B}_1\mathfrak{B}_2}[f](\omega_1,\omega_2)\in L^2(\mathbb R^2,\mathbb R),$ then the convolution theorem associated with the discrete quaternion QPFT is given by
\begin{eqnarray}
\nonumber&&\mathcal Q^{\mathbb H}_{\mathfrak{B}_1\mathfrak{B}_2}[f\star g](\omega_1,\omega_2)\\
\nonumber&&=\sqrt{N_1N_2}\Psi(i)\left[ \mathcal Q^{\mathbb H}_{\mathfrak{B}_1\mathfrak{B}_2}[f_0](\omega_1,\omega_2)\mathcal Q^{\mathbb H}_{\mathfrak{B}_1\mathfrak{B}_2}[ g](\omega_1,\omega_2)\right.\\
\nonumber&&\quad+ i\mathcal Q^{\mathbb H}_{\mathfrak{B}_1\mathfrak{B}_2}[f_1](\omega_1,\omega_2)\mathcal Q^{\mathbb H}_{\mathfrak{B}_1\mathfrak{B}_2}[ g](\omega_1,\omega_2)+j\mathcal Q^{\mathbb H}_{\mathfrak{B}_1\mathfrak{B}_2}[f_2](\omega_1,\omega_2)\mathcal Q^{\mathbb H}_{\mathfrak{B}_1\mathfrak{B}_2}[ g](\omega_1,\omega_2)\\
\nonumber&&\qquad\qquad+\left. k\mathcal Q^{\mathbb H}_{\mathfrak{B}_1\mathfrak{B}_2}[f_3](\omega_1,\omega_2)\mathcal Q^{\mathbb H}_{\mathfrak{B}_1\mathfrak{B}_2}[ g](\omega_1,\omega_2)\right]\Psi(j),\\
\end{eqnarray}
where $\Psi(i)=e^{i(c_1\omega_1^2\vartriangle u_1^2+e_1\omega_1\vartriangle u_1)}$ and $\Psi(j)=e^{j(c_2\omega_2^2\vartriangle u_2^2+e_2\omega_2\vartriangle u_2)}$
\end{theorem}
\begin{proof}
From equations (\ref{eqn dqqpft})and (\ref{eqn con}), we have
\begin{eqnarray*}
&&\mathcal Q^{\mathbb H}_{\mathfrak{B}_1\mathfrak{B}_2}[f\star g](\omega_1,\omega_2)\\
&&=\frac{1}{\sqrt{N_1}}\frac{1}{\sqrt{N_2}}\sum_{\xi_1=0}^{N_1-1}\sum_{\xi_2=0}^{N_2-1}e^{-i(a_1\xi_1^2\vartriangle t_1^2+\frac{2\pi}{N_1} \xi_1\omega_1+c_1\omega_1^2\vartriangle u_1^2+d_1\xi_1\vartriangle t_1+e_1\omega_1\vartriangle u_1)}\{f\star g\}(\xi_1,\xi_2)\\
   &&\qquad\qquad\qquad\qquad\qquad\times e^{-j(a_2\xi_2^2\vartriangle t_2^2+\frac{2\pi}{N_2} \xi_2\omega_2+c_2\omega_2^2\vartriangle u_2^2+d_2\xi_2\vartriangle t_2+e_2\omega_2\vartriangle u_2)}\\
   &&=\frac{1}{\sqrt{N_1}}\frac{1}{\sqrt{N_2}}\sum_{\xi_1=0}^{N_1-1}\sum_{\xi_2=0}^{N_2-1}e^{-i(a_1\xi_1^2\vartriangle t_1^2+\frac{2\pi}{N_1} \xi_1\omega_1+c_1\omega_1^2\vartriangle u_1^2+d_1\xi_1\vartriangle t_1+e_1\omega_1\vartriangle u_1)}\\
   &&\qquad\qquad\times \sum_{z_1=0}^{N_1-1}\sum_{z_2=0}^{N_2-1}e^{-i2a_1z_1(z_1-\xi_1)\vartriangle t_1^2}f(z_1,z_2)g(\xi_1-z_1,\xi_2-z_2)e^{-j2a_2z_2(z_2-\xi_2)\vartriangle t_2^2}\\
   &&\qquad\qquad\qquad\qquad\qquad\times e^{-j(a_2\xi_2^2\vartriangle t_2^2+\frac{2\pi}{N_2} \xi_2\omega_2+c_2\omega_2^2\vartriangle u_2^2+d_2\xi_2\vartriangle t_2+e_2\omega_2\vartriangle u_2)}\\
   &&=\frac{1}{\sqrt{N_1}}\frac{1}{\sqrt{N_2}}\sum_{z_1+y_1=0}^{N_1-1}\sum_{z_2+y_2=0}^{N_2-1}\sum_{z_1=0}^{N_1-1}\sum_{z_2=0}^{N_2-1}\\
   &&\times e^{-i[a_1(z_1^2+y_1^2+2z_1y_1)\vartriangle t_1^2+\frac{2\pi}{N_1} (z_1+y_1)\omega_1+c_1\omega_1^2\vartriangle u_1^2+d_1(z_1+y_1)\vartriangle t_1+e_1\omega_1\vartriangle u_1]} e^{i2a_1z_1y_1\vartriangle t_1^2} f(z_1,z_2)\\
   &&\times g(y_1,y_2)e^{j2a_2z_2y_2\vartriangle t_2^2}e^{-j[a_2(z_2^2+y_2^2+2z_2y_2)\vartriangle t_2^2+\frac{2\pi}{N_2} (z_2+y_2)\omega_2+c_2\omega_2^2\vartriangle u_2^2+d_2(z_2+y_2)\vartriangle t_2+e_2\omega_2\vartriangle u_2]}\\
   &&=\frac{1}{\sqrt{N_1-z_1}}\frac{1}{\sqrt{N_2-z_2}}\sum_{y_1=0}^{N_1-z_1-1}\sum_{y_2=0}^{N_2-z_2-1}\sum_{z_1=0}^{N_1-1}\sum_{z_2=0}^{N_2-1}\\
    &&\times e^{-i[a_1(z_1^2+y_1^2)\vartriangle t_1^2+\frac{2\pi}{N_1} (z_1+y_1)\omega_1+c_1\omega_1^2\vartriangle u_1^2+d_1(z_1+y_1)\vartriangle t_1+e_1\omega_1\vartriangle u_1]} f(z_1,z_2)\\
    &&\times g(y_1,y_2)e^{-j[a_2(z_2^2+y_2^2)\vartriangle t_2^2+\frac{2\pi}{N_2} (z_2+y_2)\omega_2+c_2\omega_2^2\vartriangle u_2^2+d_2(z_2+y_2)\vartriangle t_2+e_2\omega_2\vartriangle u_2]}\\
    &&=\sum_{z_1=0}^{N_1-1}\sum_{z_2=0}^{N_2-1}e^{-i(a_1z_1^2\vartriangle t_1^2+\frac{2\pi}{N_1} z_1\omega_1+d_1z_1\vartriangle t_1)}\\
    &&\times\frac{1}{\sqrt{N_1-z_1}}\frac{1}{\sqrt{N_2-z_2}}\sum_{y_1=0}^{N_1-z_1-1}\sum_{y_2=0}^{N_2-z_2-1} e^{-i(a_1y_1^2\vartriangle t_1^2+\frac{2\pi}{N_1} y_1\omega_1+c_1\omega_1^2\vartriangle u_1^2+d_1y_1\vartriangle t_1+e_1\omega_1\vartriangle u_1)} \\
    &&\qquad\qquad\qquad\times f(z_1,z_2) g(y_1,y_2) e^{-j(a_2y_2^2\vartriangle t_2^2+\frac{2\pi}{N_2} y_2\omega_2+c_2\omega_2^2\vartriangle u_2^2+d_2y_2\vartriangle t_2+e_2\omega_2\vartriangle u_2)}\\
  &&\qquad\qquad\qquad\qquad\qquad\times  e^{-j(a_2z_2^2\vartriangle t_2^2+\frac{2\pi}{N_2} z_2\omega_2+d_2z_2\vartriangle t_2)}\\
\end{eqnarray*}
Now taking $f=f_0+if_1+jf_2+kf_3,$ above equation gives
\begin{eqnarray}\label{eqn 4.4}
\nonumber&&\mathcal Q^{\mathbb H}_{\mathfrak{B}_1\mathfrak{B}_2}[f\star g](\omega_1,\omega_2)\\
 \nonumber &&=\sum_{z_1=0}^{N_1-1}\sum_{z_2=0}^{N_2-1}e^{-i(a_1z_1^2\vartriangle t_1^2+\frac{2\pi}{N_1} z_1\omega_1+d_1z_1\vartriangle t_1)}\\
    \nonumber&&\times\frac{1}{\sqrt{N_1-z_1}}\frac{1}{\sqrt{N_2-z_2}}\sum_{y_1=0}^{N_1-z_1-1}\sum_{y_2=0}^{N_2-z_2-1} e^{-i(a_1y_1^2\vartriangle t_1^2+\frac{2\pi}{N_1} y_1\omega_1+c_1\omega_1^2\vartriangle u_1^2+d_1y_1\vartriangle t_1+e_1\omega_1\vartriangle u_1)} \\
    \nonumber&&\qquad\qquad\qquad\times [f_0(z_1,z_2)+if_1(z_1,z_2)+jf_2(z_1,z_2)+kf_3(z_1,z_2)] g(y_1,y_2) \\
  \nonumber&&\times e^{-j(a_2y_2^2\vartriangle t_2^2+\frac{2\pi}{N_2} y_2\omega_2+c_2\omega_2^2\vartriangle u_2^2+d_2y_2\vartriangle t_2+e_2\omega_2\vartriangle u_2)}  e^{-j(a_2z_2^2\vartriangle t_2^2+\frac{2\pi}{N_2} z_2\omega_2+d_2z_2\vartriangle t_2)}\\
  \nonumber&&=\sum_{z_1=0}^{N_1-1}\sum_{z_2=0}^{N_2-1}e^{-i(a_1z_1^2\vartriangle t_1^2+\frac{2\pi}{N_1} z_1\omega_1+d_1z_1\vartriangle t_1)}[f_0(z_1,z_2)+if_1(z_1,z_2)+jf_2(z_1,z_2)+kf_3(z_1,z_2)]\\
    \nonumber&&\times\left\{\frac{1}{\sqrt{N_1-z_1}}\frac{1}{\sqrt{N_2-z_2}}\sum_{y_1=0}^{N_1-z_1-1}\sum_{y_2=0}^{N_2-z_2-1} e^{-i(a_1y_1^2\vartriangle t_1^2+\frac{2\pi}{N_1} y_1\omega_1+c_1\omega_1^2\vartriangle u_1^2+d_1y_1\vartriangle t_1+e_1\omega_1\vartriangle u_1)}\right. \\
   \nonumber &&\qquad\qquad\qquad\times\left. g(y_1,y_2)e^{-j(a_2y_2^2\vartriangle t_2^2+\frac{2\pi}{N_2} y_2\omega_2+c_2\omega_2^2\vartriangle u_2^2+d_2y_2\vartriangle t_2+e_2\omega_2\vartriangle u_2)}  \right\}\\
  \nonumber&&\qquad\qquad\qquad\qquad\qquad\qquad\times  e^{-j(a_2z_2^2\vartriangle t_2^2+\frac{2\pi}{N_2} z_2\omega_2+d_2z_2\vartriangle t_2)}\\
\end{eqnarray}
Using Definition \ref{def dqqpft}, (\ref{eqn 4.4}) can be rewritten as
\begin{eqnarray}\label{eqn 4.5}
\nonumber&&\mathcal Q^{\mathbb H}_{\mathfrak{B}_1\mathfrak{B}_2}[f\star g](\omega_1,\omega_2)\\
\nonumber&&=\sum_{z_1=0}^{N_1-1}\sum_{z_2=0}^{N_2-1}e^{-i(a_1z_1^2\vartriangle t_1^2+\frac{2\pi}{N_1} z_1\omega_1+d_1z_1\vartriangle t_1)}[f_0(z_1,z_2)+if_1(z_1,z_2)+jf_2(z_1,z_2)+kf_3(z_1,z_2)]\\
\nonumber&&\qquad\qquad\qquad\times  \mathcal Q^{\mathbb H}_{\mathfrak{B}_1\mathfrak{B}_2}[ g](\omega_1,\omega_2)\times  e^{-j(a_2z_2^2\vartriangle t_2^2+\frac{2\pi}{N_2} z_2\omega_2+d_2z_2\vartriangle t_2)}\\
\nonumber&&=e^{i(c_1\omega_1^2\vartriangle u_1^2+e_1\omega_1\vartriangle u_1)}\sum_{z_1=0}^{N_1-1}\sum_{z_2=0}^{N_2-1}e^{-i(a_1z_1^2\vartriangle t_1^2+\frac{2\pi}{N_1} z_1\omega_1+c_1\omega_1^2\vartriangle u_1^2+d_1z_1\vartriangle t_1+e_1\omega_1\vartriangle u_1)}\\
\nonumber&&\qquad\times [f_0(z_1,z_2)+if_1(z_1,z_2)+jf_2(z_1,z_2)+kf_3(z_1,z_2)] \mathcal Q^{\mathbb H}_{\mathfrak{B}_1\mathfrak{B}_2}[ g](\omega_1,\omega_2)\\
\nonumber&&\qquad\times  e^{-j(a_2z_2^2\vartriangle t_2^2+\frac{2\pi}{N_2} z_2\omega_2+c_2\omega_2^2\vartriangle u_2^2+d_2z_2\vartriangle t_2+e_2\omega_2\vartriangle u_2)} e^{j(c_2\omega_2^2\vartriangle u_2^2+e_2\omega_2\vartriangle u_2)}\\\\
\nonumber&&=\sqrt{N_1N_2}e^{i(c_1\omega_1^2\vartriangle u_1^2+e_1\omega_1\vartriangle u_1)}\\
\nonumber&&\qquad\qquad\times\left\{\frac{1}{\sqrt{N_1}}\frac{1}{\sqrt{N_2}}\sum_{z_1=0}^{N_1-1}\sum_{z_2=0}^{N_2-1}e^{-i(a_1z_1^2\vartriangle t_1^2+\frac{2\pi}{N_1} z_1\omega_1+c_1\omega_1^2\vartriangle u_1^2+d_1z_1\vartriangle t_1+e_1\omega_1\vartriangle u_1)}\right.\\
\nonumber&&\qquad\qquad\qquad\qquad\times [f_0(z_1,z_2)+if_1(z_1,z_2)+jf_2(z_1,z_2)+kf_3(z_1,z_2)]\\
\nonumber&&\qquad\qquad\qquad\qquad\qquad\qquad\qquad\left.\times e^{-j(a_2z_2^2\vartriangle t_2^2+\frac{2\pi}{N_2} z_2\omega_2+c_2\omega_2^2\vartriangle u_2^2+d_2z_2\vartriangle t_2+e_2\omega_2\vartriangle u_2)}\right\} \\
\nonumber&&\qquad\qquad\qquad\qquad\qquad\times \mathcal Q^{\mathbb H}_{\mathfrak{B}_1\mathfrak{B}_2}[ g](\omega_1,\omega_2)
  e^{j(c_2\omega_2^2\vartriangle u_2^2+e_2\omega_2\vartriangle u_2)}\\
\end{eqnarray}
On simplyfing, (\ref{eqn 4.5}) yields
\begin{eqnarray}
\nonumber&&\mathcal Q^{\mathbb H}_{\mathfrak{B}_1\mathfrak{B}_2}[f\star g](\omega_1,\omega_2)\\
\nonumber&&=\sqrt{N_1N_2}e^{i(c_1\omega_1^2\vartriangle u_1^2+e_1\omega_1\vartriangle u_1)}\left[ \mathcal Q^{\mathbb H}_{\mathfrak{B}_1\mathfrak{B}_2}[f_0](\omega_1,\omega_2)\mathcal Q^{\mathbb H}_{\mathfrak{B}_1\mathfrak{B}_2}[ g](\omega_1,\omega_2)\right.\\
\nonumber&&\quad+ i\mathcal Q^{\mathbb H}_{\mathfrak{B}_1\mathfrak{B}_2}[f_1](\omega_1,\omega_2)\mathcal Q^{\mathbb H}_{\mathfrak{B}_1\mathfrak{B}_2}[ g](\omega_1,\omega_2)+j\mathcal Q^{\mathbb H}_{\mathfrak{B}_1\mathfrak{B}_2}[f_2](\omega_1,\omega_2)\mathcal Q^{\mathbb H}_{\mathfrak{B}_1\mathfrak{B}_2}[ g](\omega_1,\omega_2)\\
\nonumber&&\qquad\qquad+\left. k\mathcal Q^{\mathbb H}_{\mathfrak{B}_1\mathfrak{B}_2}[f_3](\omega_1,\omega_2)\mathcal Q^{\mathbb H}_{\mathfrak{B}_1\mathfrak{B}_2}[ g](\omega_1,\omega_2)\right]e^{j(c_2\omega_2^2\vartriangle u_2^2+e_2\omega_2\vartriangle u_2)}.
\end{eqnarray}
This completes the proof.
\end{proof}

\begin{remark} {Observations:}
\begin{itemize}
  \item If we use the parameter set $\mathfrak{B}_s=(-a_s/2b_s,1/b_s,-d_s/2b_s,0,0), s=1,2$ the Theorem \ref{thm con} boils down to the convolution theorem in the DQLCT \cite{DQLCT} as
\begin{eqnarray}
\nonumber&&\mathcal Q^{\mathbb H}_{\mathfrak{B}_1\mathfrak{B}_2}[f\star g](\omega_1,\omega_2)\\
\nonumber&&=\sqrt{N_1N_2}e^{-\frac{id_1}{2b_1}m^2_1\vartriangle u_1^2}\left[ \mathcal L^{\mathbb H}_{\mathfrak{B}_1\mathfrak{B}_2}[f_0](\omega_1,\omega_2)\mathcal Q^{\mathbb H}_{\mathfrak{B}_1\mathfrak{B}_2}[ g](\omega_1,\omega_2)\right.\\
\nonumber&&\quad+ i\mathcal Q^{\mathbb H}_{\mathfrak{B}_1\mathfrak{B}_2}[f_1](\omega_1,\omega_2)\mathcal Q^{\mathbb H}_{\mathfrak{B}_1\mathfrak{B}_2}[ g](\omega_1,\omega_2)+j\mathcal Q^{\mathbb H}_{\mathfrak{B}_1\mathfrak{B}_2}[f_2](\omega_1,\omega_2)\mathcal Q^{\mathbb H}_{\mathfrak{B}_1\mathfrak{B}_2}[ g](\omega_1,\omega_2)\\
\nonumber&&\qquad\qquad+\left. k\mathcal Q^{\mathbb H}_{\mathfrak{B}_1\mathfrak{B}_2}[f_3](\omega_1,\omega_2)\mathcal Q^{\mathbb H}_{\mathfrak{B}_1\mathfrak{B}_2}[ g](\omega_1,\omega_2)\right]e^{-\frac{jd_2}{2b_2}m^2_2\vartriangle u_2^2}.
\end{eqnarray}
\item Taking $\mathfrak{B}_s=(-\cot\theta_s/2,\csc\theta_s,-\cot\theta_s/2,0,0), s=1,2$ the Theorem \ref{thm con} gives  the convolution theorem  for the novel quaternion discrete fractional Fourier transform.
 \item Taking $\mathfrak{B}_s=(0,1,0,0,0), s=1,2$ the Theorem \ref{thm con} reduces to the convolution theorem in the classical quaternion discrete Fourier transform.
 
\end{itemize}
\end{remark}

\subsection{Fast algorithm of DQQPFT}
This subsection will cover the Fast algorithm of DQQPFT through the decomposition of the quaternion signal \cite{202x}. This approach is crucial for engineering applications.\\
The DQQPFT of 2D quaternion-valued signal $f(\xi_1,\xi_2)$ can be written as
\begin{eqnarray}
\nonumber&&\mathcal Q^{\mathbb H}_{\mathfrak{B}_1\mathfrak{B}_2}[ f](\omega_1,\omega_2)\\
\nonumber&&=\frac{1}{\sqrt{N_1}}\frac{1}{\sqrt{N_2}}\sum_{\xi_1=0}^{N_1-1}\sum_{\xi_2=0}^{N_2-1}e^{-i(a_1\xi_1^2\vartriangle t_1^2+\frac{2\pi}{N_1} \xi_1\omega_1+c_1\omega_1^2\vartriangle u_1^2+d_1\xi_1\vartriangle t_1+e_1\omega_1\vartriangle u_1)}f(\xi_1,\xi_2)\\
\nonumber&&\qquad\qquad\qquad\qquad\qquad\times e^{-j(a_2\xi_2^2\vartriangle t_2^2+\frac{2\pi}{N_2} \xi_2\omega_2+c_2\omega_2^2\vartriangle u_2^2+d_2\xi_2\vartriangle t_2+e_2\omega_2\vartriangle u_2)}\\
\nonumber&&=\frac{1}{\sqrt{N_1}}\frac{1}{\sqrt{N_2}}e^{-i(c_1\omega_1^2\vartriangle u_1^2+e_1\omega_1\vartriangle u_1)}\sum_{\xi_1=0}^{N_1-1}\sum_{\xi_2=0}^{N_2-1}e^{-i(a_1\xi_1^2\vartriangle t_1^2+\frac{2\pi}{N_1} \xi_1\omega_1+d_1\xi_1\vartriangle t_1)}f(\xi_1,\xi_2)\\
\nonumber&&\qquad\qquad\qquad\qquad\qquad\times e^{-j(a_2\xi_2^2\vartriangle t_2^2+\frac{2\pi}{N_2} \xi_2\omega_2+d_2\xi_2\vartriangle t_2)}e^{-j(c_2\omega_2^2\vartriangle u_2^2e_2\omega_2\vartriangle u_2)}\\
\nonumber&&=\frac{1}{\sqrt{N_1}}\frac{1}{\sqrt{N_2}}E^i_{\mathfrak{B}_1}\sum_{\xi_1=0}^{N_1-1}\sum_{\xi_2=0}^{N_2-1}e^{-i(a_1\xi_1^2\vartriangle t_1^2+\frac{2\pi}{N_1} \xi_1\omega_1+d_1\xi_1\vartriangle t_1)}f(\xi_1,\xi_2)\\
\nonumber&&\qquad\qquad\qquad\qquad\qquad\times e^{-j(a_2\xi_2^2\vartriangle t_2^2+\frac{2\pi}{N_2} \xi_2\omega_2+d_2\xi_2\vartriangle t_2)}E^j_{\mathfrak{B}_2}\\
\label{5.1}&&=\frac{1}{\sqrt{N_1}}\frac{1}{\sqrt{N_2}}E^i_{\mathfrak{B}_1}\sum_{\xi_1=0}^{N_1-1}\sum_{\xi_2=0}^{N_2-1}e^{-i(\frac{2\pi}{N_1} \xi_1\omega_1)}\psi(\xi_1,\xi_2) e^{-j(\frac{2\pi}{N_2} \xi_2\omega_2)}E^j_{\mathfrak{B}_2}
\end{eqnarray}
where $E^i_{\mathfrak{B}_1}=e^{-i(c_1\omega_1^2\vartriangle u_1^2+e_1\omega_1\vartriangle u_1)},$ \qquad  $E^j_{\mathfrak{B}_2}= e^{-j(c_2\omega_2^2\vartriangle u_2^2e_2\omega_2\vartriangle u_2)}$ and 
\begin{equation}\label{5.2}
\psi(\xi_1,\xi_2)=e^{-i(a_1\xi_1^2\vartriangle t_1^2+d_1\xi_1\vartriangle t_1)}f(\xi_1,\xi_2)e^{-j(a_2\xi_2^2\vartriangle t_2^2+d_2\xi_2\vartriangle t_2)}
\end{equation}

Consider 
\begin{equation}\label{5.3}
\Psi_c(\omega_2,\omega_1)=\sum_{\xi_1=0}^{N_1-1}\sum_{\xi_2=0}^{N_2-1}e^{-i(\frac{2\pi}{N_1} \xi_1\omega_1)}\psi(\xi_1,\xi_2) e^{-j(\frac{2\pi}{N_2} \xi_2\omega_2)}
\end{equation}

Thus, we have 
\begin{equation}\label{5.4}
\frac{1}{2}\left[\Psi_c(\omega_2,\omega_1)+\Psi_c(-\omega_2,\omega_1)\right]=\sum_{\xi_1=0}^{N_1-1}\sum_{\xi_2=0}^{N_2-1}e^{-i(\frac{2\pi}{N_1} \xi_1\omega_1)}\psi(\xi_1,\xi_2)\cos\left(\frac{2\pi}{N_2} \xi_2\omega_2\right)
\end{equation}

and 

\begin{equation}\label{5.5}
\frac{1}{2}\left[\Psi_c(\omega_2,\omega_1)-\Psi_c(-\omega_2,\omega_1)\right]=\sum_{\xi_1=0}^{N_1-1}\sum_{\xi_2=0}^{N_2-1}(-i)e^{-i(\frac{2\pi}{N_1} \xi_1\omega_1)}\psi(\xi_1,\xi_2)\sin\left(\frac{2\pi}{N_2} \xi_2\omega_2\right).
\end{equation}

From (\ref{5.4})and (\ref{5.5}), we have
\begin{eqnarray*}
\nonumber&&\frac{1}{2}\left[\Psi_c(\omega_2,\omega_1)+\Psi_c(-\omega_2,\omega_1)\right]-\frac{1}{2}k\left[\Psi_c(\omega_2,\omega_1)-\Psi_c(-\omega_2,\omega_1)\right]\\
&&\qquad\qquad\qquad\qquad=\sum_{\xi_1=0}^{N_1-1}\sum_{\xi_2=0}^{N_2-1}e^{-i(\frac{2\pi}{N_1} \xi_1\omega_1)}\psi(\xi_1,\xi_2)e^{-j(\frac{2\pi}{N_2} \xi_2\omega_2)}
\end{eqnarray*}
Equivalently, 
\begin{equation}\label{5.6}
\sum_{\xi_1=0}^{N_1-1}\sum_{\xi_2=0}^{N_2-1}e^{-i(\frac{2\pi}{N_1} \xi_1\omega_1)}\psi(\xi_1,\xi_2)e^{-j(\frac{2\pi}{N_2} \xi_2\omega_2)}=\frac{1}{2}\left[(1-k)\Psi_c(\omega_2,\omega_1)+(1+k)\Psi_c(-\omega_2,\omega_1)\right]
\end{equation}
By virtue of  (\ref{5.6}), eqn (\ref{5.1}) becomes

\begin{eqnarray}
\nonumber&&\mathcal Q^{\mathbb H}_{\mathfrak{B}_1\mathfrak{B}_2}[ g](\omega_1,\omega_2)\\
\label{5.7}&&\qquad\qquad=\frac{1}{2\sqrt{N_1N_2}}E^i_{\mathfrak{B}_1}\left[(1-k)\Psi_c(\omega_2,\omega_1)+(1+k)\Psi_c(-\omega_2,\omega_1)\right]E^j_{\mathfrak{B}_2}
\end{eqnarray}
Thus in order to compute discrete QQPFT, we need to compute the complex 2D DFT of the signal $\psi(\xi_1, \xi_2)$ as (\ref{5.4}) and (\ref{5.5}), while noting that the signal $\psi(\xi_1, \xi_2)$ is quaternion-valued. So, we first need to decompose the signal $\psi(\xi_1, \xi_2)$ as
\begin{eqnarray}\label{5.8}
\nonumber \psi(\xi_1,\xi_2)&=&\psi_0(\xi_1,\xi_2)+i\psi_1(\xi_1,\xi_2)+j\psi_2(\xi_1,\xi_2)+k\psi_3(\xi_1,\xi_2)\\
\nonumber&=&[\psi_0(\xi_1,\xi_2)+i\psi_1(\xi_1,\xi_2)]+j[\psi_2(\xi_1,\xi_2)-i\psi_3(\xi_1,\xi_2)]\\
\nonumber&=&\tilde\psi(\xi_1,\xi_2)+j\hat\psi(\xi_1,\xi_2).\\
\end{eqnarray}
Hence (\ref{5.3}), takes the form

\begin{eqnarray}
\nonumber&&\Psi_c(\omega_2,\omega_1)\\
\nonumber&&=\sum_{\xi_1=0}^{N_1-1}\sum_{\xi_2=0}^{N_2-1}e^{-i(\frac{2\pi}{N_1} \xi_1\omega_1)}[\tilde\psi(\xi_1,\xi_2)+j\hat\psi(\xi_1,\xi_2)] e^{-j(\frac{2\pi}{N_2} \xi_2\omega_2)}\\
\nonumber&&=\sum_{\xi_1=0}^{N_1-1}\sum_{\xi_2=0}^{N_2-1}e^{-i(\frac{2\pi}{N_1} \xi_1\omega_1)}\tilde\psi(\xi_1,\xi_2) e^{-j(\frac{2\pi}{N_2} \xi_2\omega_2)}+\sum_{\xi_1=0}^{N_1-1}\sum_{\xi_2=0}^{N_2-1}e^{-i(\frac{2\pi}{N_1} \xi_1\omega_1)}j\hat\psi(\xi_1,\xi_2) e^{-j(\frac{2\pi}{N_2} \xi_2\omega_2)}\\
\nonumber&&=\sum_{\xi_1=0}^{N_1-1}\sum_{\xi_2=0}^{N_2-1}e^{-i(\frac{2\pi}{N_1} \xi_1\omega_1)}\tilde\psi(\xi_1,\xi_2) e^{-j(\frac{2\pi}{N_2} \xi_2\omega_2)}+j\sum_{\xi_1=0}^{N_1-1}\sum_{\xi_2=0}^{N_2-1}e^{i(\frac{2\pi}{N_1} \xi_1\omega_1)}\hat\psi(\xi_1,\xi_2) e^{-j(\frac{2\pi}{N_2} \xi_2\omega_2)}\\
\label{5.9}&&=\mathcal F[\tilde\psi(\xi_1,\xi_2)](\omega_1,\omega_2)+j\mathcal F[\hat\psi(-\xi_1,\xi_2)](\omega_1,\omega_2),
\end{eqnarray}
where $\mathcal F[\psi]$ is 2D complex discrete FT.\\
 Thus with the help of  (\ref{5.9}), we can compute (\ref{5.7}) and hence the discrete quaternion quadratic phase Fourier transform is obtained with the help of 2D  complex discrete FT.\\

From the above discussion,
we observe that the computation of the discrete QQPFT corresponds to the
following four steps :
\begin{itemize}
\item For a given signal $f(\xi_1,\xi_2),$ calculate $\psi(\xi_1,\xi_2)$ by using (\ref{5.2}).\\
    
    \item Using (\ref{5.8}), decompose signal $\psi(\xi_1,\xi_2).$\\
    
    \item Using (\ref{5.9}), evaluate $\Psi_c(\omega_2,\omega_1).$\\
    
    \item Finally, using (\ref{5.7}), we calculate $\mathcal Q^{\mathbb H}_{\mathfrak{B}_1\mathfrak{B}_2}[ f](\omega_1,\omega_2).$
\end{itemize}

\section{Conclusion}\label{sec6}

This article defines the 2D two-sided DQQPFT, which is generalization of  the 2D two-sided DQFT and is required to compute the QQPFT using digital methods. The 2D DQQPFT's basic properties are listed. Additionally, consideration is given to the Plancherel theorem, convolution theorem, and reconstruction formula. The secret to using 2D DQFT in engineering is its fast algorithm. The fast algorithm for 2D DQQPFT method is obtained for this. Lastly, we have demonstrated how the DQQPFT can be used to examine discrete versions of linear time-varying systems.
\section*{Declarations}
\begin{itemize}
\item  Availability of data and materials: The data is provided on the request to the authors.
\item Competing interests: The authors have no competing interests.
  \item Use of AI tools declaration:
The authors declare they have not used Artificial Intelligence (AI) tools in the creation of this article.
\item Funding: No funding was received for this work
\item Author's contribution: All the authors equally contributed towards this work.

\end{itemize}

{\bf{References}}

\end{document}